\documentclass{amsart}

\usepackage{amssymb,color}
\usepackage{amsfonts}
\usepackage{amsmath}
\usepackage{euscript}
\usepackage{enumerate}
\usepackage{graphics}

\newtheorem{theorem}{Theorem}[section]
\newtheorem{lemma}[theorem]{Lemma}
\newtheorem{note}[theorem]{Note}
\newtheorem{prop}[theorem]{Proposition}
\newtheorem{cor}[theorem]{Corollary}

\newtheorem*{Theorem1'}{Theorem 1'}

\theoremstyle{definition}

\theoremstyle{remark}

\numberwithin{equation}{section}

\newcommand \End{{\mathrm {End}}}
\newcommand \dm{{\mathrm {dim}}}
\newcommand \chr{{\mathrm {char}}}
\newcommand \gl{{\mathfrak {gl}}}
\renewcommand \sl{{\mathfrak {sl}}}
\newcommand \so{{\mathfrak {o}}}
\newcommand \sy{{\mathfrak {sp}}}

\newcommand \GL{{\mathrm {GL}}}

\newcommand \si{{\sigma}}
\newcommand \B{{\mathrm {Bil}}}

\begin{document}

\title {Reductivity of the Lie algebra of a bilinear form}

\author{S. Ruhallah Ahmadi}
\address{Department of Mathematics and Statistics, Univeristy of Regina, Canada}
\email{seyed.ruhala.ahmadi@gmail.com}

\author{Martin Chaktoura}
\email{martin\_chaktoura@yahoo.com.ar}

\author{Fernando Szechtman}
\email{fernando.szechtman@gmail.com}






\begin{abstract}
Let $f:V\times V\to F$ be a totally arbitrary bilinear form
defined on a finite dimensional vector space $V$ over a a field
$F$, and let $L(f)$ be the subalgebra of $\gl(V)$ of all
skew-adjoint endomorphisms relative to $f$. Provided $F$ is
algebraically closed of characteristic not 2, we determine all
$f$, up to equivalence, such that $L(f)$ is reductive. As a
consequence, we find, over an arbitrary field, necessary and
sufficient conditions for $L(f)$ to be simple, semisimple or
isomorphic to $\sl(n)$ for some $n$.
\end{abstract}

\maketitle

\section{Introduction}

As is well-known \cite{J} the complex Lie algebra associated to a
non-degenerate symmetric or skew-symmetric bilinear form is
semisimple, but for $D_1$, and simple, except for $D_1$ and $D_2$.
Moreover, $\sl(n)$ can be defined by means of one these forms only
when $n=2,4$. We wonder what becomes of these statements if we
allow totally arbitrary forms and fields. That is, given a field
$F$ and a bilinear form $f:V\times V\to F$, when is the Lie
algebra $L(f)$ simple or semisimple, and when is $L(f)$ isomorphic
to $\sl(n)$ for some $n$? In this paper we answer all of these
questions.

Throughout this paper we understand a Lie algebra to be reductive
if every solvable ideal is central. For finite dimensional Lie
algebras over a field of characteristic~0 this coincides with the
classical notion of reductivity, i.e., the adjoint representation
is completely reducible. Note, however, that $\gl(n)$ is reductive
according to our definition, except only when $\chr(F)=2=n$, but
it is so in the classical sense only when $\chr(F)\nmid n$ (see
\cite{B}, Chapter 1, \S 6, Exercise 24).

Returning to our prior questions, note that $\sl(n)$ fails to be
simple when $\chr(F)|n$, but is always reductive, except only when
$\chr(F)=2=n$ (see \cite{B}, Chapter 1, \S 6, Exercise 24). Thus
we orient our search of the bilinear forms $f$ such that
$L(f)\cong \sl(n)$ by first finding, in characteristic not 2, all
$f$ such that $L(f)$ is reductive, and stays reductive after
extending scalars to an algebraic closure of $F$, i.e., $L(f)$ is
absolutely reductive, like $\sl(n)$. This may be of interest on
its own right and plays, in fact, a decisive role in answering the
above questions.

It is necessary to divide the analysis of the reductivity of
$L(f)$ in various cases, depending on the nature of $f$. This
analysis can be found in Propositions \ref{jodd3}, \ref{p23},
\ref{zxc}, \ref{sumario1} and \ref{sumario2}. A summary of these
results when $F$ is algebraically closed of characteristic not 2
is given in Theorem \ref{todjur}. A consequence of Theorem
\ref{todjur} is that over an algebraically closed field of
characteristic not 2 the only $f$ such that $L(f)$ is simple
modulo its center, like $\sl(n)$, are in fact those for which
$L(f)$ is already simple.

Simiplicity or semisimiplicity are such strong conditions that
they only occur under optimal circumstances, as described in Theorems \ref{fsemi} and \ref{panda}.

By combining the above information Theorem \ref{tip} determines
all cases when $\sl(n)\cong L(f)$ in characteristic not 2.
The case $\chr(F)=2$ requires separate investigations, as seen
in Theorem \ref{pe4} for $n>2$, and Proposition \ref{pe5} for $n=2$.

In addressing the reductivity of $L(f)$, we first deal with the case
when $f$ is degenerate. The study of $L(f)$ in this case makes use of
a family of $L(f)$-invariant subspaces of~$V$, as found in
\cite{DS}. We also need to know \cite{G} all indecomposable
degenerate bilinear forms.

When $f$ is non-degenerate our initial approach makes use of the
asymmetry of~$f$ and its basic properties, as developed in
\cite{R}. The classification \cite{W} of all linear
operators that are the asymmetry of some bilinear form is required
as well.

A more detailed investigation of $L(f)$ relies on a list of
representatives \cite{HS} for the indecomposable bilinear forms
over algebraically closed fields of characteristic not~2. In
particular, we compute the structure of $L(f)$ for all $f$
indecomposable. Propositions \ref{x1} and \ref{x2} show that, for some indecomposable $f$, the Lie algebra
$L(f)$ is essentially isomorphic to the current truncated Lie
algebra $\gl(2)\otimes F[X]/(X^m)$, a fact that is not at all
clear from the definition of $L(f)$. For interesting recent work
on current Lie algebras and their representations see \cite{CR},
\cite{BW}.

The structure of the Lie algebra $L(f)$ requires considerably more
information than the structure of its subalgebras $L(f_1),\dots
L(f_m)$, where $f=f_1\perp\cdots\perp f_m$ and each $f_i$ is
indecomposable. This is already evident if $f$ is non-degenerate
and symmetric, in which case $L(f)$ is a classical Lie algebra,
while each $L(f_i)$ reduces to~0 in characteristic not 2. A
thorough study of the structure of the group $G(f)$ preserving $f$
is carried out in \cite{D}, \cite{S}.



\section{Preliminaries}
\label{sectermnot}

We fix throughout a field $F$ of characteristic $\ell$, a non-zero
finite dimensional vector space $V$ over $F$, and a bilinear form
$f:V\times V\to F$.

If $B=\{v_1,\dots,v_n\}$ is a basis of $V$ the Gram matrix $S\in
M_n(F)$ of $f$ relative to~$B$ is defined by $S_{ij}=f(v_i,v_j)$.
If $g:W\times W\to F$ is also a bilinear form, then $f$ and $g$
are said to be equivalent if they admit the same Gram matrix.

The left and right radicals of $f$ are respectively defined by
$$
l(f)=\{u\in V\,|\, f(u,V)=0\},\; r(f)=\{u\in V\,|\, f(V,u)=0\}.
$$
It is not difficult to see that the following statements are
equivalent: $l(f)=0$; $r(f)=0$; $f$ admits an invertible Gram
matrix. When this happens we say that $f$ is non-degenerate. The
radical of $f$, denoted by $\mathrm{Rad}(f)$, is the intersection
of the left and right radicals of $f$.



We write $L(f)$ for the subalgebra of $\gl(V)$ of all $x\in\gl(V)$
that are skew-adjoint with respect to $f$, i.e.,
$$
L(f)=\{x\in\gl(V)\,|\, f(xv,w)=-f(v,xw)\text{ for all }v,w\in V\}.
$$
Likewise, for $S\in\gl(n)$ let $L(S)$ stand for the subalgebra of
all $X\in\gl(n)$ satisfying $$X' S+SX=0,$$ where $X'$ is the
transpose of $X$. Let $M_B:\gl(V)\to\gl(n)$ be the isomorphism
that assigns to each $x\in\gl(V)$ its matrix $M_B(x)$ relative to
a basis $B$ of $V$. Then $M_B$ sends $L(f)$ onto $L(S)$, so
$L(f)\cong L(S)$.

Given subspaces $U_1,\dots U_m$ of $V$ we write
\begin{equation}
\label{desco}
V=U_1\perp\cdots \perp U_m\text{ or }f=f_{U_1}\perp\cdots\perp f_{U_m}
\end{equation}
to mean that $V=U_1\oplus\cdots\oplus U_m$ and $f(U_i,U_j)=0$ for
all $1\leq i\neq j\leq m$, where $f_{U_i}$ denotes the restriction
of $f$ to $U_i$.

If $V$ admits no decomposition $V=U\perp W$ except when $U=0$ or
$W=0$ we say that $f$ is indecomposable.

Let $J_n(\lambda)$ stand for the lower
triangular Jordan block with eigenvalue $\lambda\in F$, and write $J_n$ for
$J_n(0)$. It is known \cite{G},
\cite{DS2} that there is one and only one
indecomposable degenerate bilinear form, up to equivalence,
defined on a vector space of dimension $n$, namely one admitting $J_n$ as Gram matrix.

In general, it is clear that there exist subspaces $U_1,\dots,U_m$
of $V$ such that (\ref{desco}) holds and each $f_{U_i}$ is indecomposable.
Let $V_{\mathrm{even}}$ (resp. $V_{\mathrm{odd}}$) be the sum of all $U_i$ such that
$f_{U_i}$ has Gram matrix $J_n$ with $n$ even (resp. odd). Let
 $V_{\mathrm{ndeg}}$ be the sum of all other $U_j$. Then the sizes,
 including multiplicities, of all $J_n$ so arising, as well as
 the equivalence type of $f_{\mathrm{ndeg}}$, are uniquely
 determined by $f$ (see \cite{G}, \cite{DS}, \cite{DS2}). Here $f_{\mathrm{even}},f_{\mathrm{odd}},f_{\mathrm{ndeg}}$
 stand for the restrictions of $f$ to $V_{\mathrm{even}},
 V_{\mathrm{odd}}, V_{\mathrm{ndeg}}$, respectively, and
 $f_{\mathrm{ndeg}}$ is non-degenerate.

It is important to note \cite{DS} that $V_{\mathrm{even}},
 V_{\mathrm{odd}}, V_{\mathrm{ndeg}}$ are not uniquely determined
 subspaces of $V$ and, in particular, they are not
 $L(f)$-invariant. However, when $V_{\mathrm{odd}}=0$ then
 $V_{\mathrm{even}}$ and $V_{\mathrm{ndeg}}$ can be intrinsically
 defined from $f$ and they are $L(f)$-invariant.

If $f$ is non-degenerate its asymmetry $\si\in\GL(V)$ is defined by
$$
f(w,v)=f(v,\si w),\quad v,w\in V.
$$

\begin{theorem}\label{sergei} Suppose that $\ell\neq 2$ and that $F$ is algebraically closed.
Then an indecomposable non-degenerate bilinear form $f$ admits one
and only one of the following as Gram matrix, except for a
possible interchange of $\lambda$ and $\lambda^{-1}$:
$$
A_n(\lambda)=\left(
    \begin{array}{cc}
      0 & I_{n} \\
      J_{n}(\lambda) & 0 \\
    \end{array}
  \right),\lambda\neq (-1)^{n+1},
$$
$$
\Gamma_1=(1),\,\Gamma_2=\left(\begin{array}{cc}
0&-1\\
1&1
\end{array}\right),\,\Gamma_3= \left(\begin{array}{ccc}
0&0&1\\
0&-1&-1\\
1&1&0
\end{array}\right),\,\Gamma_4= \left(\begin{array}{cccc}
0&0&0&-1\\
0&0&1&1\\
0&-1&-1&0\\
1&1&0&0
\end{array}\right),\dots
$$
In the case of $A_n(\lambda)$ the asymmetry of $f$ has two
elementary divisors, namely $(X-\lambda)^n,(X-\lambda^{-1})^n$,
and in the case of $\Gamma_n$ a single elementary divisor, namely
$(X-1)^n$ if $n$ is odd and $(X+1)^n$ if $n$ is even.
\end{theorem}

\begin{proof} It is known \cite{R} that under the stated assumptions on $F$ the equivalence type of
a non-degenerate bilinear form and the similarity type of its asymmetry determine each other. The result now
follows from the classification, given in \cite{W}, Theorem 2.3.1, of all invertible operators that can possibly
be an asymmetry.
\end{proof}

The above list of representatives is taken from \cite{HS}.

\section{The eigenvalues of certain linear operators}
\label{seceigop}

For $A\in M_m (F)$ and $B\in M_n (F)$ we consider the linear maps
$$l_A, r_B :M_{m,n}(F)\to M_{m,n}(F)$$
given by $l_A (C)=AC,\;\;r_B (C)=CB$. Note that $l_A$ and $r_B$ commute.

This section finds the eigenvalues of $l_A r_B$ from those of $A$
and $B$, and uses this information to determine
$L(f_{\mathrm{even}})$. The form of these eigenvalues is
well-known \cite{F} and is a special case of Property $P$ of a
pair of matrices~\cite{SC}. We also look at a sharpened
description of these eigenvalues, which may be of independent
interest.

\begin{lemma}\label{eigenvalue}
Suppose all eigenvalues of $A$ and $B$ lie in $F$. Then all
eigenvalues of $l_A r_B$ lie in $F$. Furthermore, any such
eigenvalue is of the form $\alpha\beta$, where $\alpha$ is an
eigenvalue of $A$ and $\beta$ is an eigenvalue of $B$.
\end{lemma}
\begin{proof}
Since $A$ and $l_A$ have the same minimal polynomial, all
eigenvalues of $l_A$ are in $F$, and likewise those of $r_B$.
Since $l_A$ and $r_B$ commute, they are simultaneously
triangularizable, and the result follows.
\end{proof}

We next sharpen Lemma \ref{eigenvalue} by means of the
Jordan-Chevalley decomposition. Recall that if $x\in\End(V)$ has
all eigenvalues in $F$ there are unique $d,n\in\End(V)$ such that
$x=d+n$; $d$ is diagonalizable; $n$ is nilpotent; $dn=nd$.
Moreover, $d,n$ are polynomials in $x$ with coefficients in $F$.
We refer to $x=d+n$ as the JC-decomposition of $x$.

The lists of eigenvalues given in Lemma \ref{eigenvalue2} below
allow for possible repetitions.

\begin{lemma}\label{eigenvalue2}
Suppose $A$ has eigenvalues $\alpha_1 ,\dots,\alpha_m$ and $B$ has
eigenvalues $\beta_1,\dots,\beta_n$, all in $F$. Then the $n\times m$
eigenvalues of $l_Ar_B$ in $F$ are
\[\alpha_1\beta_1,\dots,\alpha_1\beta_n,\alpha_2\beta_1\dots,\alpha_2\beta_n,\dots,\alpha_m\beta_1,\dots,\alpha_m\beta_n.
\]
\end{lemma}
\begin{proof} Let $A=S+N$ and $B=D+M$ be the JC-decompositions of $A$ and $B$.

\subsection*{Step 1} The JC-decomposition
of $l_Ar_B$ is
$$l_Ar_B=(l_Sr_D)+(l_Sr_M+l_Nr_D+l_Nr_M).$$

Indeed, since $S$ and $l_S$ have the same minimal polynomials and
$S$ is diagonalizable, so must be $l_S$. Likewise, $r_D$ is
diagonalizable. Since $l_S$ and $r_D$ commute, they are
simultaneously diagonalizable, so $l_Sr_D$ is diagonalizable. In addition,
$l_A$ and $r_B$ have JC-decompositions
$$l_A=l_S+l_N,\;\;\;r_B=r_D+r_M.$$
Since $l_A$ and $r_B$ commute and $l_S, l_N$ are polynomials in
$l_A$, it follows that $l_S, l_N$ commute with $r_B$, and hence
with its polynomials $r_D,r_M$. Therefore $l_S, l_N, l_D, l_M$
commute pairwise. Since $l_N,r_M$ are nilpotent,
$l_Sr_M+l_Nr_D+l_Nr_M $ is nilpotent.
\subsection*{Step 2}
We may assume that $A=S$ and $B=D$ are diagonalizable.

This is because both members of each of the following 3 pairs have
the same eigenvalues: $(A,S)$, $(B,D)$, $(l_Ar_B,l_Sr_D)$.
\subsection*{Step 3} We may assume that $A=\mathrm{diag}(\alpha_1,\dots,\alpha_m)$.

Indeed, there is $P\in GL_m(F)$ such that
$A^*=P^{-1}AP=\mathrm{diag}(\alpha_1,\dots,\alpha_m)$. Let
$B^*=P^{-1}BP$. Then both members of each pair $(A,A^*)$,
$(B,B^*)$, $(l_A r_B,l_{A^*}r_{B^*})$ have the same eigenvalues,
since for $Q=r_{P^{-1}}l_P$ we have
$$
l_{A^*}r_{B^*}=l_{P^{-1}}l_A l_P r_P r_B r_{P^{-1}}= l_{P^{-1}}r_P
l_Ar_B r_{P^{-1}} l_P =Q^{-1} l_A r_B Q.
$$
\subsection*{Step 4} We now conclude the proof under the assumptions that
$B$ is diagonalizable and
$A=\mathrm{diag}(\alpha_1,\dots,\alpha_m)$. Since $B$ and $B'$ are
similar, $B'$ is diagonalizable with eigenvalues
$\beta_1,\dots.\beta_n$. Let $\{v_1,\dots,v_n\}$ be a basis of the
column space $F^n$ formed by eigenvectors of $B'$ with eigenvalues
$\beta_1,\dots.\beta_n$. From $B' v_j =\beta_j v_j$ we infer $v_j'
B=\beta_j v_j'$ for all $1\leq j\leq n$. For $1\leq i\leq m$ and
$1\leq j\leq n$ let $C_{ij}$ be the $m\times n$ matrix all of
whose rows are 0 except for row $i$ which is equal to $v_j'$. Then
$$
AC_{ij}B=\alpha_i\beta_j C_{ij},
$$
and the result follows.
\end{proof}

For $0\neq p\in F[X]$, let $p^* \in F[X]$ be its adjoint
polynomial, defined by $$p^*(X)=X^{\mathrm{deg} (p)}p(1/X).$$
Thus, if $p(X)=a_0+a_1X+\cdots+a_{k-1}X^{k-1}+a_kX^k$ with $a_k\ne
0$, we have
$$
p^*(X) = a_0X^k + a_1X^{k-1} + \cdots + a_{k-1}X + a_k.
$$
It is easy to see that $(p_1p_2)^*=p_1^* p_2^*$, and if $p(0)\neq
0$ then $p^{**}=p$.

For $A\in M_m (F)$, we write $p_A$ for the minimal polynomial
of~$A$.

\begin{lemma}
\label{eigenvalue3} Let $A\in M_m (F)$ and $B\in M_n (F)$. Then 1
is an eigenvalue of $l_A r_B$ if and only if $p_A$ and $p^*_B$ are
not relatively prime.
\end{lemma}
\begin{proof}
Suppose 1 is an eigenvalue of $l_A r_B$. By Lemma
\ref{eigenvalue}, we have $1=\alpha\beta$, where $\alpha, \beta$
are eigenvalues of $A,B$ in some extension $K$ of $F$. Thus
$\alpha$ is a common root of $p_A$ and $p^*_B$ in $K$, so $p_A$
and $p^*_B$ are not relatively prime (in $F[X]$). Suppose
conversely that $p_A$ and $p^*_B$ are not relatively prime. There
is an extension $K$ of~$F$ and $\alpha\in K$ such that $\alpha\neq
0$, $p_A (\alpha)=0$ and $p_B (1/\alpha)=0$, so Lemma
\ref{eigenvalue2} yields that $1=\alpha\times 1/\alpha$ is an eigenvalue
of $l_A r_B$.
\end{proof}

\begin{cor}
\label{central} Let $A\in\gl(n)$, suppose $\gcd(p_A,p_A^*)=1$ and
set
$$
T=\left(%
\begin{array}{cc}
  0 & A \\
  I_n & 0\\
\end{array}%
\right).
$$
Then $L(T)\cong C_{\gl(n)}(A)$, the centralizer of $A$ in
$\gl(n)$.
\end{cor}

\begin{proof} Let
$$
X=\left(%
\begin{array}{cc}
  X_1 & X_2 \\
  X_3 & X_4 \\
\end{array}%
\right),
$$
where each $X_i\in\gl(n)$. Then $X\in L(T)$ if and only if
$$
AX_3=-X_3',X_1=-X_4', \;AX_4=-X_1'A,\; -X_2'A=X_2.
$$
From $AX_3=-X_3'$ we get $X_3'A'=-X_3$, so
$$
AX_3A'=X_3,
$$
i.e.,
$$
l_A r_{A'}(X_3)=X_3.
$$
Since $\gcd(p_{A},p_A^*)=1$, Lemma \ref{eigenvalue2} implies
$X_3=0$. Likewise we see that $X_2=0$. Thus $L(T)$ consists of all
$$
X=\left(%
\begin{array}{cc}
  -Y' & 0 \\
  0 & Y \\
\end{array}%
\right),
$$
where $Y\in \gl(n)$ and $YA=AY$, as required.
\end{proof}

As an application of Corollary \ref{central} we next find $L(f)$
when $V=V_{\mathrm{even}}$.

\begin{cor}
\label{3cr} Suppose that $f$ has Gram matrix
$$
S=J_{2r_1}\oplus\cdots\oplus J_{2r_t}
$$
relative to some basis $B$ of $V$. Let $n=r_1+\cdots+r_t$ and set
$$
A=J_{r_1}\oplus\cdots\oplus J_{r_t}.
$$
Then $L(f)\cong C_{\gl(n)}(A)$.
\end{cor}

\begin{proof} Reordering $B$ we obtain a basis $C$ of $V$ relative to which $f$
has Gram matrix
$$
T=\left(%
\begin{array}{cc}
  0 & A \\
  I_n & 0 \\
\end{array}%
\right).
$$
Since $p_A^*=1$, Corollary \ref{central} yields $L(f)\cong
C_{\gl(n)}(A)$.
\end{proof}

\section{Disposing of $V_{\mathrm{odd}}$}

In this section we show that in characteristic not 2 the Lie
algebra $L(f)$ is reductive if and only if either $f=0$, in which
case $L(f)\cong \gl(n,F)$ for $n=\dm(V)$, or else
$V_{\mathrm{odd}}=0$ and $L(f)\cong L(f_{\mathrm{even}})\oplus
L(f_{\mathrm{ndeg}})$, with reductive summands.

We refer the reader to \cite{DS}, \S 2 and \S 3, for the
definition and basic properties of the subspaces
$L^1(V),L^3(V),\dots$ and $R^1(V),R^3(V),\dots$, as well as
${}^\infty\! V$ of $V$, all of which are $L(f)$-invariant.

\begin{lemma}
\label{jodd} Suppose that $f$ has Gram matrix $S=J_{2n+1}$, $n\geq
1$, relative to some basis $B=\{e_1,\dots,e_{2n+1}\}$ of $V$. Let
$$
K=\{x\in L(f)\,|\, x ^{\infty}\!V=0,\; xV\subseteq
^{\infty}\!\!\!V \}.
$$
Then $K$ is an abelian ideal of $L(f)$ of dimension $n$.
Moreover, $K=L(f)\cap\sl(V)$ and there is $y\in L(f)$ such that
$L(f)=\langle y\rangle\ltimes K$, where $[y,x]=2x$ for all $x\in
K$.
\end{lemma}

\begin{proof} Reordering $B$ we obtain the basis
$C=\{e_1,e_3,\dots,e_{2n+1},e_2,e_4,\dots,e_{2n}\}$, relative to
which $f$ has Gram matrix
$$
T=\left(%
\begin{array}{cc}
  0 & J \\
  M & 0 \\
\end{array}%
\right),
$$
where $J$ has size $(n+1)\times n$ and is obtained by adding a
zero row on top of $I_n$, and $M$ has size $n\times (n+1)$ and is
obtained by adding a zero column at the end of $I_n$. Let
$$
X=\left(%
\begin{array}{cc}
  X_1 & X_2 \\
  X_3 & X_4 \\
\end{array}%
\right)
$$
be partitioned as $T$ and let $x\in \gl(V)$ be the operator
represented by $X$ with respect to $C$.

The subspaces $L^1(V),L^3(V),\dots$ and $R^1(V),R^3(V),\dots$ are
$L(f)$-invariant and are explicitly described in \cite{DS},
Theorem 3.1. Suppose $x\in L(f)$. Then $X_3=0$ with $X_1$ both
upper and lower triangular, i.e. $X_1$ diagonal. Moreover,
$$
X_1'J+JX_4=0,\; X_4'M+M X_1=0,\;X_2'J+MX_2=0.
$$
The first two conditions imply that $X_4$ is also diagonal and, in
fact, $X_1=a I_{n+1}$, $X_4=-aI_n$ for some $a\in F$. Furthermore,
the $X\in L(T)$ satisfying $X_1=0$ represent the $x\in L(f)$ that
belong to $K$, which is clearly an abelian ideal of $L(f)$. Let
$y\in L(f)$ have matrix $Y$, where $Y_1=I_{n+1}$ and $Y_2=0$. Then
$[y,x]=2x$ for all $x\in K$ and $L(f)=\langle y\rangle\ltimes K$.
It remains to show $\dm(K)=n$. This follows easily from
$X_2'J+MX_2=0$, which means the following: all diagonals of $X_2$
parallel to the two main diagonals have equal entries; diagonals
of $X_2$ equidistant to the line passing through the middle of the
two main diagonals have opposite entries.
\end{proof}

\begin{lemma}
\label{joddunoymedio} Suppose that $L(f)$ is reductive and
$\mathrm{Rad}(f)\neq 0$. Then $f=0$.
\end{lemma}

\begin{proof} Suppose, if possible, that $f\neq 0$, i.e. $\mathrm{Rad}(f)\neq V$. Let
$$
V=W\oplus U
$$
be a decomposition of $V$, where $W=\mathrm{Rad}(f)$ and $U\neq
0$. Let
$$
K=\{x\in\gl(V)\,|\, xW=0\text{ and }xV\subseteq W\}.
$$
Now for $x\in\gl(V)$ the condition $xV\subseteq W$ implies $x\in
L(f)$, so $K$ is an abelian ideal of $L(f)$. Moreover, any $x\in
\gl(W)$ can be extended in an obvious manner to an element of
$L(f)$. Let $x=1_W$, viewed as an element of $L(f)$, and let $y\in
K$ be any non-zero linear map that sends $W$ to 0 and $U$ to $W$.
Then
$$
[x,y]=xy=y\neq 0,
$$
contradicting the reductivity of $L(f)$.
\end{proof}

\begin{lemma}
\label{jodd2} Suppose that $L(f)$ is reductive,
$V_{\mathrm{odd}}\neq 0$ and $\ell\neq 2$. Then $f=0$.
\end{lemma}

\begin{proof} Let $f|_{V_{\mathrm{odd}}}$ have Gram matrix
$$
S=J_{2r_1+1}\oplus\cdots\oplus J_{2r_t+1},\quad r_1\geq\cdots\geq r_t\geq 0,
$$
relative to a basis $B$ of $V_\mathrm{odd}$. Suppose, if possible,
that $r_1\geq 1$. Let $C$ consist of the first $2r_1+1$ vectors of
$B$. We set
$$Y=\langle C\rangle,\quad g=f|_{Y\times Y}.$$
By Lemma \ref{jodd} there is $x\in L(g)$ such that $x
^{\infty}\!Y=0$, $x Y\subseteq  ^{\infty}\!\!\!Y$ and $x\notin
Z(L(g))$. Viewing $L(g)$ as a subalgebra of $L(f)$, the properties of $L(V)$ and
$R(V)$ given in \cite{DS} ensure
that $x ^{\infty}\!V=0$, $xV\subseteq  ^{\infty}\!\!\!V$ and
$x\notin Z(L(f))$. Thus $r_1=\cdots=r_t=0$, so
$\mathrm{Rad}(f)\neq 0$. Now apply Lemma \ref{joddunoymedio}.
\end{proof}

\begin{prop}
\label{jodd3} Suppose $\ell\neq 2$. Then $L(f)$ is reductive if
and only if either $f=0$, in which case $L(f)\cong \gl(n,F)$ for
$n=\dm(V)$, or else $V_{\mathrm{odd}}=0$ and
$L(f_{\mathrm{even}})$, $L(f_{\mathrm{ndeg}})$ are both reductive.
Moreover, in the latter case $L(f)\cong L(f_{\mathrm{even}})\oplus
L(f_{\mathrm{ndeg}})$.
\end{prop}

\begin{proof} If $f=0$ then $L(f)\cong \gl(n,F)$ is reductive (see \cite{B}, Chapter 1, \S 6,
Exercise 24). Suppose $f\neq 0$. If $L(f)$ is reductive then Lemma
\ref{jodd2} gives that $V_{\mathrm{odd}}=0$, in which case
$L(f)\cong L(f_{\mathrm{even}})\oplus L(f_{\mathrm{ndeg}})$ as
explained in \S\ref{sectermnot}, with both summands reductive.
Obviously, if these conditions hold, $L(f)$ is reductive.
\end{proof}

\section{Centralizers in the general Lie algebra}
\label{secredcent}

Let $x\in\gl(V)$ and set $L=C_{\gl(V)}(x)$. This section gives
necessary and sufficient conditions for $L$ to be reductive based
on the nature of $x$. This will be used later to determine the
reductivity of $L(f)$.

Much is known \cite{JN} about nilpotent orbits and centralizers of
nilpotent elements in semisimple Lie algebras over an
algebraically closed field of characteristic not 2. However,
we will not make use of this material, as we mostly deal with
arbitrary fields and, in addition, we can furnish elementary
arguments to suit all our needs.

We may use $x$ to view $V$ as a module for the polynomial algebra
$F[X]$. Thus we say that $x$ is cyclic if $V$ is cyclic as an
$F[X]$-module, which happens when the minimal and characteristic
polynomials of $x$ coincide, and we say that $x$ is semisimple if
$V$ is a semisimple $F[X]$-module, which is equivalent to the
minimal polynomial of $x$ having a multiplicity-free prime
factorization in $F[X]$.

If $p_1,\dots,p_t$ are the distinct monic irreducible factors of the
minimal polynomial of $x$ then the primary decomposition of $V$ is
$$
V=V_{p_1}\oplus\cdots\oplus V_{p_t},
$$
where for any $q\in F[X]$ we define
$$
V_q=\{v\in V\,|\, q^m\cdot v=0\text{ for some }m\geq 1\}.
$$
We abuse this notation and write
$$
V_\lambda=V_{X-\lambda},\quad \lambda\in F.
$$
Let $x_i=x|_{V_{p_i}}$ for $1\leq i\leq t$. Then
$$
C_{\gl(V)}(x)\cong C_{\gl(V_{p_1})}(x_1)\oplus\cdots\oplus
C_{\gl(V_{p_t})}(x_t).
$$
Thus $C_{\gl(V)}(x)$ is reductive if and and only if every
$C_{\gl(V_{p_i})}(x_i)$ is reductive. We may thus assume that
$V=V_p$, where $p\in F[X]$ is monic and irreducible, so that
$E=F[X]/(p)$ is a field extension of $F$.

Now $V$ has a cyclic decomposition
$$
V=F[X]v_1\oplus\cdots\oplus F[X]v_r,
$$
where the minimal polynomials of $v_1,\dots,v_r$ with respect to
$x$ are $p^{m_1},\dots,p^{m_r}$, and $m_1\geq\cdots\geq m_r$.
These are the elementary divisors of $x$. In this case, we have
the following result.

\begin{prop}
\label{redcent} The Lie algebra $L$ is reductive if and only if
either $x$ is cyclic, in which case $L$ is abelian and
$\dm(L)=\dm(V)$, or $x$ is semisimple and $(r,\ell)\neq (2,2)$, in
which case $L\cong \gl(r,E)$ as Lie algebras over $F$.
\end{prop}

\begin{proof} Note first of all that $L$ is the Lie algebra of the associative algebra
$\End_{F[X]}(V)$ of all $F[X]$-endomorphisms of $V$.

Suppose first $x$ is cyclic. It is well-known and easy to see that
in this case $\End_{F[X]}(V)=\{q(x)\,|\, q\in F[X]\}\cong
F[X]/(p^{m_1})$, so $L$ is abelian, and therefore reductive, of
dimension $\dm(V)$.

Suppose next $x$ is semisimple. By assumption there is a basis $B$
of $V$ relative to which the matrix of $x$, say $A$, is the direct
sum of $r$ copies of the companion matrix $C_p$ of $p$. Since the
only matrices commuting with $C_p$ are the polynomials in $C_p$,
we see that the centralizer of $A$, as an associative $F$-algebra,
is isomorphic to $M_r(E)$, hence its Lie algebra is $F$-isomorphic
to $\gl(r,E)$, which is reductive if and only if $(r,\ell)\neq (2,2)$.

Suppose finally that $x$ is neither cyclic nor semisimple, i.e.,
$r>1$ and $m_1>1$.

An element of $\End_{F[X]}(V)$ is nothing but an element of
$\End_{F}(V)$ satisfying $v_i\to w_i$, $1\leq i\leq r$, where
$p^{m_i}\cdot w_i=0$.

Let $W=\{v\in V\,|\, p\cdot v=0\}$. This is an $L$-invariant
subspace of $V$. Therefore
$$
J=\{g\in L\,|\, gW=0,\; gV\subseteq W\}
$$
is an ideal of $L$, clearly abelian. Consider the elements $g,h\in
L$ given by
$$
g(v_1)=p^{m_2-1}\cdot v_2,\; g(v_2)=\cdots=g(v_r)=0,
$$
$$
h(v_1)=v_1,\; h(v_2)=\cdots=h(v_r)=0.
$$
By definition $g\in J$, $gh(v_1)=p^{m_2-1}\cdot v_2\neq 0$ and
$hg(v_1)=0$. Thus $J$ is a non-central abelian ideal of $L$, so
$L$ is not reductive.
\end{proof}

\begin{note} It is easy to see that if $r>1$ and $(r,\ell)\neq (2,2)$ then
the Lie algebra $\gl(r,E)$ over~$F$ is absolutely reductive if and only if $p\in F[X]$
is separable.
\end{note}

Corollary \ref{3cr} and Proposition \ref{redcent} immediately
yield the following

\begin{prop}
\label{p23} Suppose $V=V_{\mathrm{even}}$ and let $m=\dm(V)/2$.
 Then $L(f)$ is reductive if and only if
either $f$ is indecomposable, in which case $L(f)$ is abelian of
dimension $m$, or else $f_{\mathrm{even}}$ has $m>1$ components,
all of type $J_2$, in which case $L(f)\cong \gl(m)$.
\end{prop}

\section{A canonical decomposition of $V_{\mathrm{ndeg}}$}

We assume in this section that $f$ is non-degenerate with asymmetry $\si$.
Viewing~$V$ as an $F[X]$-module via $\si$, we give a
decomposition of $L(f)$ as the direct sum of ideals associated to
the primary components of the $F[X]$-module~$V$. We also look at
the ideals of $L(f)$ corresponding to the primary components
$V_p$, where $p$ is an irreducible factor of the minimal polynomial $p_\si$
of $\si$ such that $\gcd(p,p^*)=1$.

Let us begin by recalling some of the basic properties of $\si$,
as found in \cite{R}. Observe first of all that $f$ is symmetric
if and only if $\si = 1$, and skew-symmetric if and only if $\si =
-1$.

\begin{lemma}
\label{lem21a} If $v,w\in V$ and $x\in L(f)$ then $f(\sigma v,\sigma w)=f(v,w)$ and $\sigma x=x\sigma$.
\end{lemma}
\begin{proof} Let $v,w\in V$ and $x\in L(f)$. Then
$$
f(\si v,\si w)=f(w,\si v)=f(v,w)
$$
and
$$
f(v,\si x w)=f(xw, v)=-f(w,xv)=-f(xv,\si w)=f(v, x\si w).
$$
\end{proof}

\begin{lemma}
\label{lem22a} The endomorphism $\sigma-\sigma^{-1}$ belongs to
$L(f)$.
\end{lemma}

\begin{proof} Let $v,w\in V$. Then by Lemma \ref{lem21a}
$$
f((\si-\si^{-1}) v, w)=f(\si
v,w)-f(\si^{-1}v,w)=f(v,\si^{-1}w)-f(v,\si
w)=-f(v,(\si-\si^{-1})w).
$$
\end{proof}

Let $B$ be a basis of $V$ and let $A$ and $S$ be the respective
matrices of $f$ and $\si$. Then
$$
A'=AS,\quad S=A^{-1}A'.
$$
Therefore
$$
S^{-1}=(A')^{-1}A=(A')^{-1}S'A',
$$
so $\si$ is similar to $\si^{-1}$. Recalling from \S\ref{seceigop}
the definition and basic properties of the adjoint polynomial, it
follows that
$$
p_\si^*=\pm p_\si.
$$
Thus, if $p\in F[X]$ is a factor of $p_\si$ then so must be $p^*$.
If, in addition, $p$ is irreducible then, since $p(0)\neq 0$,
$p^*$ must also be irreducible. If, in addition, $\gcd(p,p^*)\neq
1$ then $p^*=cp$ for some $c\in F$, so $p=cp^*$, whence $c=\pm 1$.
All in all, if $p\in F[X]$ is an irreducible factor of $p_\si$
then $\gcd(p,p^*)=1$ or $p^*=\pm p$.

In this regard, we have the following result \cite{R}.

\begin{lemma}
\label{lem83} Let $p$ and $q$ be irreducible factors of $p_\si$.
If $q$ is not a scalar multiple of $p^*$, then $V_p$ and $V_q$ are
orthogonal, i.e., $f(V_p,V_q)=f(V_q,V_p)=0$.
\end{lemma}

It follows that the primary decomposition of $V$ is
$$
V = (V_{p_1}\oplus V_{p_1^*}) \perp \cdots \perp (V_{p_r}\oplus
V_{p_r^*}) \perp V_{q_1}\perp \cdots \perp V_{q_s},
$$
where
$$
f(V_{p_i},V_{p_i})=f(V_{p_i^*},V_{p_i^*})=0,\quad 1\le i\le r,
$$
and $p_1,\ldots,p_r,q_1,\ldots,q_s$ are irreducible factors of
$p_\si$ such that $p_i\ne \pm p_i^*$ and $q_i=\pm q_i^*$.

Given an irreducible polynomial $p\in F[X]$ we set
$f_{p}=f|_{V_{p}+V_{p^*}}$, and further let $f_\lambda=f_{X-\lambda}$ for $\lambda\in F$.
Clearly $f_p$ is non-degenerate and
$V_{p}\cong V_{p^*}$ as vector spaces.

By Lemma \ref{lem21a}, $\si$ commutes with all $x\in L(f)$, so
every primary component of~$V$ is $L(f)$-invariant, which yields
the following decomposition of $L(f)$:
$$
L(f) = L(f_{p_1}) \oplus \cdots \oplus
L(f_{p_r}) \oplus L(f_{q_1}) \oplus
\cdots \oplus L(f_{q_s}).
$$
In particular, $L(f)$ is reductive if and only if so is every
$L(f_{p_i})$ and $L(f_{q_j})$.

\begin{lemma} Suppose $V=V_p\oplus V_{p^*}$, where $p\in F[X]$ and
$\gcd(p,p^*)=1$. Then $L(f)\cong C_{\gl(V_p)}(\sigma|_{V_p})$.
\end{lemma}

\begin{proof} Let $B = \{ v_1,\ldots,v_n \}$ be a basis of $V_p$ and let $C=\{
\eta_1,\ldots,\eta_n \}$ be the dual basis of $V_p^*$. Thus
$\eta_i(v_j)=\delta_{ij}$. Clearly, the map $V_{p^*}\to
V_{p}^*$, given by $u\mapsto f(u,-)$, is bijective. The inverse
image of $C$ under this map is a basis $D=\{u_1,\ldots,u_n\}$
of~$V_{p^*}$ satisfying $f(u_i,v_j)=\delta_{ij}$.

Let $A$ be the matrix of $\si\vert_{V_p}$ relative to ~$B$. We
claim that the Gram matrix of $f$ relative to the basis $B\cup D$
of $V$ is
$$
T=\left( \begin{matrix}
0 & A'\\
I_n & 0\\
\end{matrix} \right).
$$
Indeed, the condition $f(u_i,v_j)=\delta_{ij}$ justifies the block
$I$, while the fact that $f(V_p,V_p)=0=f(V_{p^*},V_{p^*})$
justifies the 0 blocks. As for the remaining block, note that for
every $1\leq i,j\leq n$ we have
$$
f(v_i,u_j) = f(u_j,\sigma v_i) = f \left( u_j , \sum_{k=1}^n
A_{ki}v_k \right) = \sum_{k=1}^n A_{ki} h(u_j,v_k) = \sum_{k=1}^n
A_{ki}\delta_{jk} = A_{ji}.
$$
Thus $L(f)\cong C_{\gl(n)}(A')\cong C_{\gl(n)}(A)\cong
C_{\gl(V_p)}(\sigma|_{V_p})$ by Corollary \ref{central}.
\end{proof}

We have all the information required to derive the following result.

\begin{prop}\label{zxc} Suppose $\ell\neq 2$ and $f$ is non-degenerate. Let $K$ be an algebraic closure of $F$.
Then  $L(f)$  is absolutely reductive if and only if so are
$L(f_1)$ and $L(f_{-1})$, and for every irreducible factor $p\in
F[X]$ of $p_\si$ such that $p(\pm 1)\neq 0$ either
$\si\vert_{V_p}$ is cyclic, in which case $L(f_p)$ is abelian, or
else $\si\vert_{V_p}$ is semisimple and $p$ is separable, in which
case $L(f_p)\otimes K$ is the direct sum of $\mathrm{deg}(p)$
copies of $\gl(r,K)$, where $r=\dm(V_p)/\mathrm{deg}(p)>1$.
\end{prop}



\section{Centralizers in orthogonal and symplectic Lie algebras}
\label{seccensym}

Here we recall from \cite{D} and \cite{S} how to view $L(f)$ as
the centralizer of a suitable element of an orthogonal or
symplectic Lie algebra.

Let $\B(V)$ denote the space of all bilinear forms on $V$.
Note that $\B(V)$ is a natural right $\End(V)$-module via
$$
(g\cdot a)(u,v)=g(u,av),\quad g\in\B(V),a\in \End(V),u,v\in V.
$$
For a fixed $g\in\B(V)$ the map $\End(V)\to\B(V)$ given by $a\to
g\cdot a$ is a linear isomorphism if and only if $g$ is
non-degenerate, in which case $a$ is invertible if and only if
$g\cdot a$ is non-degenerate. In this case, given any $h\in\B(V)$
we will write $a_{g,h}$ for the unique $a\in \End(V)$ such that
$h(u,v)=g(u,av)$ for all $u,v\in V$.

Let $g\in\B(V)$. Define $g',g^+$ and $g^-$ in $\B(V)$ as follows:
$g'(u,v)=g(v,u)$ for all $u,v\in V$, $g^+=g+g'$ and $g^-=g-g'$.

In this notation, if $f$ is non-degenerate then $\si=u_{f,f'}$ is
the asymmetry of $f$. If $f^+$ is non-degenerate, we write $\si^{+-}=a_{f^+,f^-}$.
If $f^-$ is non-degenerate, we write $\si^{-+}=a_{f^-,f^+}$.
We let $p_\sigma\in F[X]$ stand for the minimal polynomial of $\si$.

\begin{lemma}
\label{ze1} Suppose $f$ is non-degenerate. Then

(a) $f^-$ is non-degenerate if and only if $p_\si(1)\neq 0$.

(b) $f^+$ is non-degenerate if and only if $p_\si(-1)\neq 0$.
\end{lemma}

\begin{proof} (a) For $v,w\in V$ we have
$$
f^-(v,w)=f(v,w)-f(w,v)=f(v,w)-f(v,\si w)=f(v,(1_V-\si))w).
$$
Since $f$ is non-degenerate, it follows that $f^-$ is
non-degenerate if and only if $1_V-\si$ is invertible, which is
equivalent to $p_\si(1)\neq 0$.

(b) This follows as above mutatis mutandis.
\end{proof}

\begin{lemma}\label{ze2} (a) If $f^-$ is non-degenerate, then
$\si^{-+}\in L(f^-)$.

(b) If $f^+$ is non-degenerate, then $\si^{+-}\in L(f^+)$.
\end{lemma}

\begin{proof} (a) Let $v,w\in V$. Then
$$
f^-(\si^{-+}v,w)+f^-(v,\si^{-+}w)=-f^-(w,\si^{-+}v)+f^+(v,w)=-f^+(w,v)+f^+(v,w)=0.
$$

(b) This follows as above mutatis mutandis.

\end{proof}

\begin{lemma}\label{ze3} (a) Suppose $f^-$ is non-degenerate and $\ell\neq
2$. Then
$$
L(f)=C_{L(f^-)}(\si^{-+}).
$$

(b) Suppose $f^+$ is non-degenerate and $\ell\neq 2$. Then
$$
L(f)=C_{L(f^+)}(\si^{+-}).
$$
\end{lemma}

\begin{proof} (a) Let $x\in\gl(V)$. Since $\ell\neq 2$ we see that
$x\in L(f)$ if and only if $x\in L(f^-)$ and $x\in L(f^+)$, that
is, $x\in L(f^-)$ and $f^+(xv,w)+f^+(v,xw)=0$ for all $v,w\in V$.
Since $f^-$ is non-degenerate, in the presence of $x\in L(f^-)$,
the last condition can be replaced by
$$
\begin{aligned}
0&=f^-(xv,\si^{-+}w)+f^-(v,\si^{-+}xw)=-f^-(v,x\sigma^{-+}w)+f^-(v,\si^{-+}xw)\\
&=f^-(v,(\sigma^{-+}x-x\sigma^{-+})w),\quad v,w\in V
\end{aligned}
$$
which means that $\sigma^{-+}x=x\sigma^{-+}$, as required.

(b) This follows as above mutatis mutandis.
\end{proof}

\section{Dealing with $V_1$: first round}
\label{secuno}

Here we apply the ideas from \S\ref{seccensym} to study the
reductivity of $L(f)$ when $f$ is non-degenerate and $\si-1_V$ is
nilpotent, where $\si$ is the asymmetry of $f$.

We suppose until further notice that $f$ has Gram matrix
$$
\left(
    \begin{array}{cc}
      0 & I_{n_1} \\
      J_{n_1}(1) & 0 \\
    \end{array}
  \right)\oplus\cdots\oplus\left(
    \begin{array}{cc}
      0 & I_{n_k} \\
      J_{n_k}(1) & 0 \\
    \end{array}
  \right),\quad n_1\leq\cdots\leq n_k,
$$
relative to some basis of $V$. Reordering this basis we obtain a
new basis, say $Q$, relative to which $f$ has Gram matrix
$$
A=\left(
    \begin{array}{cc}
      0 & I \\
      J & 0 \\
    \end{array}
  \right),
$$
where all blocks have the same size $n=n_1+\cdots+n_k$ and
$$
J=J_{n_1}(1)\oplus\cdots\oplus J_{n_k}(1).
$$
Let $S$ be the matrix of  $\sigma$ relative to $Q$. Then
$$
S=A^{-1}A'=\left(
    \begin{array}{cc}
      J^{-1} & 0 \\
      0 & J' \\
    \end{array}
  \right).
$$
Thus $\sigma$ has only one eigenvalue, namely 1, and elementary
divisors
$$(X-1)^{n_1},(X-1)^{n_1},\dots,(X-1)^{n_k},(X-1)^{n_k}.
$$

We assume that $\ell\neq 2$ for the remainder of this section.
Since $p_\si(-1)\neq 0$, Lemma \ref{ze1} ensures that $f^+$ is
non-degenerate. Let $S^{+-}$ be the matrix of $\si^{+-}$ relative
to $Q$. By definition
$$
f^-(v,w)=f^+(v,\si^{+-}w),\quad v,w\in V,
$$
so
$$A-A'=(A+A')S^{+-}.
$$
Since $A+A'$ is invertible, we have
$$ S^{+-}=(A+A')^{-1}(A-A')= \left(
    \begin{array}{cc}
      (I+J)^{-1}(J-I) & 0 \\
      0 & (I+J')^{-1}(I-J') \\
    \end{array}
  \right).
$$
Performing these calculations block by block we see that
$\si^{+-}$ is nilpotent, with elementary divisors
$$X^{n_1},X^{n_1},\dots,X^{n_k}, X^{n_k}.
$$

It follows from  \S\ref{seccensym} that $S^{+-}\in L(A+A')$ and $L(f)\cong C_{L(A+A')}(S^{+-})$.
In order to use this identification more effectively, we will replace $A+A'$ and $S^{+-}$ by more convenient
matrices. Let
$$
Z=
\left(
    \begin{array}{cc}
      (I+J)^{-1} & 0 \\
      0 & I \\
    \end{array}
  \right)
$$
and set
$$
C=Z'(A+A')Z=\left(
    \begin{array}{cc}
      0 & I \\
      I & 0 \\
    \end{array}
  \right),
$$
$$
T=Z^{-1}S^{+-}Z=
\left(
    \begin{array}{cc}
      R_1 & 0 \\
      0 & R_2 \\
    \end{array}
  \right),
$$
where $R_1,R_2\in\gl(n)$, $n=n_1+\cdots+n_k$, are similar to
$$L=J_{n_1}(0)\oplus\cdots\oplus J_{n_k}(0).$$

Now the map $L(A+A')\to L(C)$ given by $U\mapsto Z^{-1}UZ$ is a
Lie isomorphism, so $L(f)\cong C_{L(C)}(T)$, where $T\in
L(C)=\so(2n)$, the usual matrix version of the orthogonal Lie
algebra. In particular, $R_2=-R_1'$.

Now for any $G\in\GL_n(F)$ the matrix
$$
H=\left(
    \begin{array}{cc}
      G & 0 \\
      0 & (G')^{-1} \\
    \end{array}
  \right)
$$
belongs to the orthogonal group $O(2n)$, i.e., the isometry group
of $C$. Since $R_1$ is similar to $L$ we may choose $G\in\GL_n(F)$
so that $G^{-1}R_1 G=L$. Thus $V\mapsto H^{-1}VH$ is an
automorphism of $\so(2n)$ mapping $T\in\so(2n)$ to
$$
Y=\left(
    \begin{array}{cc}
      L & 0 \\
      0 & -L' \\
    \end{array}
  \right)\in\so(2n).
$$
All in all, there is a basis $B$ of $V$ relative to which $L(f)$
is represented by $C_{\so(2n)}(Y)$. To compute with $L(f)$ we look
for all matrices
$$N=\left(
    \begin{array}{cc}
      P & D \\
      E & -P' \\
    \end{array}
  \right)\in\so(2n)
$$
that commute with $Y$. Since $D,E$ are skew-symmetric, this is
equivalent to
$$
 LP=PL,\quad
LD=(LD)',\quad EL=(EL)'.
$$


\begin{prop}
\label{x1}
 Suppose that $\ell\neq 2$, that $n$ is even, and that $f$ has Gram
 matrix $A_n(1)$ relative to some basis of $V$. Then $L(f)$ has dimension $2n$ and
is isomorphic to the current truncated Lie algebra $\gl(2)\otimes
F[X]/(X^{n/2})$. In particular, $L(f)$ is reductive if and only if
$n=2$. Moreover, let $W=\ker(\si^{+-})^2$ and set
$$
K=\{x\in L(f)\,|\, xV\subseteq W\text{ and }x W=0\}.
$$
Then $K$ is an abelian ideal of $L(f)$ not contained in its
center, provided $n>2$.
\end{prop}

\begin{proof} We keep the above notation and note that $L=J_n(0)$. Observe that if $n=4$ the matrices $N$ have the form
$$
\left(\begin{array}{cccc|cccc}
a &0 &0 &0 &0 &0 &0 &j\\
b &a &0 &0 &0 &0 &-j&0\\
c &b &a &0 &0 &j &0 &k\\
d &c &b &a &-j&0 &-k&0\\\hline
0 &e &0 &g &-a &-b&-c &-d\\
-e&0 &-g&0 &0 &-a &-b&-c\\
0 &g &0 &0 &0 &0 &-a&-b\\
-g&0 &0 &0 &0 & 0& 0&-a\\
\end{array}\right)
$$

In general, $P$ is lower triangular with equal entries along any
sub-diagonal. By an even (resp. odd) diagonal of $P$ we will
understand a sub-diagonal lying at an even (resp. odd) distance
from the main diagonal.

A typical $D$ is lower triangular with respect to the secondary
diagonal (the one joining positions $(1,n)$ and $(n,1)$), with
alternating entries $j,-j,\dots,j,-j$ along the secondary diagonal
as well as along all sub-diagonals of the the secondary diagonal
lying at even distance from it. We will be refer to these as even
diagonals of $D$. Sub-diagonals of the secondary diagonal of $D$
at an odd distance from it must be~0.

A typical $E$ is is upper triangular with respect to the secondary diagonal,
with alternating entries
along the secondary diagonal as well as
all its super-diagonals at even distance from it. We will be refer to these as
even diagonals of $E$. Super-diagonals
of the secondary diagonal of $E$ at an odd distance from it must be 0.

Let $h_0$ be the element associated to the main diagonal of $P$,
that is
$$
h_0=e_{11}+\cdots+e_{nn}-(e_{n+1,n+1}+\cdots+e_{2n,2n}).
$$
Let $e_0,f_0$ be the elements likewise associated to the secondary diagonals of $D$
and $E$, respectively.

Next let $h_1$ be the element associated to the next even diagonal of
$P$, and let $e_1,f_1$ be associated to the next non-zero diagonals
of $D$ and $E$, respectively. Likewise we define $h_i,e_i,f_i$ for all
$0\leq i< n/2$, and agree that $h_i,e_i,f_i$ are 0 for $i\geq n/2$.

Let $M$ be the span of all elements associated to the odd diagonals of $P$.

We see that the following relations hold in $L=C_{\so(2n)}(Y)$:
$$
[M,L]=0,[h_i,h_j]=0,[h_i,e_j]=2e_{i+j},[h_i,f_j]=-2f_{i+j},[e_i,f_j]=h_{i+j}.
$$
This proves all but the last assertion. To see this, let
$B=\{u_1,\dots,u_n,v_1,\dots,v_n\}$. The matrix of $\si^{+-}$
relative to $B$ is $Y$, so $W$ is spanned by
$u_{n-2},u_{n-1},v_1,v_2$. Suppose $n\geq 4$. Then
$h_{\frac{n}{2}-1}\in K$. Since
$[h_{\frac{n}{2}-1},e_0]=2e_{\frac{n}{2}-1}$, we see that
$h_{\frac{n}{2}-1}\in K$ is not central. As $K$ is an abelian
ideal of $L(f)$, the proof is complete.
\end{proof}

\begin{cor}
\label{xcor} Suppose that $\ell\neq 2$ and $f$ has at least one
indecomposable component with Gram matrix $A_n(1)$, $n>2$ even.
Then $L(f)$ is not reductive.
\end{cor}

\begin{proof} We may assume without loss of generality that $f$ is
non-degenerate and $\si-1_V$ is nilpotent. In this case $f^+$ is
non-degenerate, so $\si^{+-}$ is well-defined.

Let $W=\ker(\si^{+-})^2$ and set $K=\{x\in L(f)\,|\, xV\subseteq
W\text{ and }x W=0\}$. It is an abelian ideal of $L(f)$. The fact
that it is not central follows by looking at the subalgebra
$L(A_n(1))$ considered in Proposition \ref{x1}.
\end{proof}

\begin{cor}
\label{casoextra} Suppose that $\ell\neq 2$ and $f$ has at least
two indecomposable components with Gram matrix $A_2(1)$. Then
$L(f)$ is not reductive.
\end{cor}

\begin{proof} Suppose first that the Gram matrix $f$ relative to
some basis is $A_2(1)\oplus A_2(1)$. Let $W=\ker\si^{+-}$ and set
$K=\{x\in L(f)\,|\, xV\subseteq W\text{ and }x W=0\}$. Let $x,y\in
L(f)$ be represented by matrices $N$ with blocks $P$ respectively
equal to
$$
\left(%
\begin{array}{cccc}
  0 & 0 & 0 & 0 \\
  1 & 0 & 0 & 0 \\
  0 & 0 & 0 & 0 \\
  0 & 0 & 0 & 0 \\
\end{array}%
\right)\text{ and }\left(%
\begin{array}{cccc}
  0 & 0 & 1 & 0 \\
  0 & 0 & 0 & 1  \\
  0 & 0 & 0 & 0 \\
  0 & 0 & 0 & 0 \\
\end{array}%
\right)
$$
relative to the basis $B$ used prior to Proposition \ref{x1}. Then
$x\in K$ and $[x,y]\neq 0$. The general case follows easily from
this one.
\end{proof}

\begin{note} An elementary matrix representation of a current
truncated Lie algebra $L\otimes F[X]/(X^n)$ can be obtained from a
representation $R:L\to\gl(m)$, as follows. Set $M=R(L)$ and for
$A_0,\dots, A_{n-1}\in M$  consider the block upper triangular
matrix $M(A_0,\dots,A_{n-1})\in\gl(mn)$ given by:
$$
\left(\begin{array}{ccccc}
A_0 & A_1   & A_2    & \cdots & A_{n-1}\\
  & \ddots  & \ddots & \ddots &  \vdots       \\
  &         & \ddots & \ddots &  A_2  \\
  &         &        & \ddots &  A_1  \\
  &         &        &        &  A_0
\end{array}\right)
$$
Then $a_0\otimes 1+\cdots+a_{n-1}\otimes X^{n-1}\mapsto
M(R(a_0),\dots,R(a_{n-1}))$ is a representation of $L\otimes
F[X]/(X^n)$.
\end{note}

\section{Dealing with $V_{-1}$: first round}\label{secmenosuno}

We suppose here that $\ell\neq 2$. As in \S\ref{secuno}, we wish
to study the reductivity of $L(f)$ when $f$ is non-degenerate and
$\si+1_V$ is nilpotent, where $\si$ is the asymmetry of $f$. We
assume until further notice that $f$ has Gram matrix $A_n(-1)$
relative to some basis of $V$. Let $S$ be the matrix of $\sigma$
in this basis. Then
$$
S=A^{-1}A'=\left(
    \begin{array}{cc}
      J^{-1} & 0 \\
      0 & J' \\
    \end{array}
  \right),
$$
so $\sigma$ has only one eigenvalue, namely -1, and elementary
divisors $(X+1)^n$, $(X+1)^n$.


Moreover, since $p_\si(1)\neq 0$ now $f^-$ is non-degenerate. Let $S^{-+}$ be the matrix of $\si^{-+}$ relative to our basis. Then
$$
S^{-+}=(A-A')^{-1}(A+A'),
$$
which is nilpotent with elementary divisors $X^n,X^n$. Reasoning
as in \S\ref{secuno} we see that there is a basis $B$ of $V$
relative to which $L(f)$ is represented by $C_{\sy(2n)}(Y)$, where
$$
C=\left(
    \begin{array}{cc}
      0 & I \\
      -I & 0 \\
    \end{array}
  \right),
$$
$\sy(2n)=L(C)$ and
$$
 Y=\left(
    \begin{array}{cc}
      J_n(0) & 0 \\
      0 & -J_n(0)' \\
    \end{array}
  \right)\in\sy(2n).
$$
\begin{prop}
\label{x2}
 Suppose that $\ell\neq 2$, that $n$ is odd, and that $f$ has Gram matrix
 $A_n(-1)$ relative to some basis of $V$. Then $L(f)$ has dimension $2n+1$
and is isomorphic to the current truncated Lie algebra
$\gl(2)\otimes F[X]/(X^{(n+1)/2})$ modulo a one dimensional
central ideal. In particular, $L(f)$ is reductive if and only if
$n=1$. Moreover, let $W=\ker\si^{+-}$ and set
$$
K=\{x\in L(f)\,|\, xV\subseteq W\text{ and }x W=0\}.
$$
Then $K$ is an abelian ideal of $L(f)$ not contained in its
center, provided $n>1$.
\end{prop}

\begin{proof} This goes much as the proof of Proposition \ref{x1} mutatis mutandis. Explicitly, $D$ and $E$ are now symmetric, so
the conditions to commute with $Y$ become
$$
J_n(0)P=PJ_n(0),\quad J_n(0)D=-(J_n(0)D)',\quad EJ_n(0)=-(EJ_n(0))'.
$$
For instance, if $n=5$ we obtain the matrices
$$
\left(\begin{array}{ccccc|ccccc}
a &0 &0 &0 &0 &0  &0 &0  &0 &j\\
b &a &0 &0 &0 &0  &0 &0  &-j&0\\
c &b &a &0 &0 &0  &0 &j  &0 &k\\
d &c &b &a &0 &0  &-j&0  &-k&0\\
e &d &c &b &a &j  &0 &k  &0 &l\\\hline
g &0 &h &0 &i &-a &-b&-c &-d&-e\\
0 &-h&0 &-i &0&0  &-a&-b &-c&-d\\
h &0 &i &0 &0 &0  &0 &-a &-b&-c\\
0 &-i&0 &0 &0 &0  &0 &0  &-a&-b\\
i &0 &0 &0 &0 &0  &0 &0  &0 &-a\\
\end{array}\right)
$$
The general description of the blocks $P,D,E$ is identical to the
one given in Proposition \ref{x1}.
The non-zero elements $h_i,e_i,f_i$ are defined for $0\leq i\leq
(n-1)/2$. These changes take care of all assertions but the last.
For this, notice that now $W$ is spanned by $u_n,v_1$. Suppose
$n\geq 3$. Then $h_{(n-1)/2}\in K$. Since
$[h_{(n-1)/2},e_0]=2e_{(n-1)/2}$, it follows that $h_{(n-1)/2}\in
K$ is not central.
\end{proof}

Exactly as in \S\ref{secuno} we now obtain the following

\begin{cor}
\label{xcor2} Suppose that $\ell\neq 2$ and $f$ has at least one
indecomposable component with Gram matrix $A_n(-1)$, $n>1$ odd.
Then $L(f)$ is not reductive.
\end{cor}

\section{Dealing with $V_1$ and $V_{-1}$: second round}

We assume throughout this section that $\ell\ne 2$.


\begin{lemma}\label{lega} Suppose $f$ has Gram matrix $\Gamma_n$ relative to a basis of
$V$. Then $f$ is non-degenerate and the corresponding matrix of
the asymmetry $\sigma$ is
$$
S_n = \Gamma_n^{-1}\Gamma_n' = \epsilon \left( \begin{matrix}
1       & -2        & 2         &  \cdots   & 2\epsilon\\
0       & 1         & -2        &  \ddots   &\vdots\\
0       & 0         & 1         & \ddots    &  \vdots\\
\vdots  & \vdots    & \vdots    & \ddots        & -2\\
0       & 0         & 0         &\cdots     & 1
\end{matrix} \right),\, \epsilon=(-1)^{n+1},
$$
which is cyclic with minimal polynomial $(X-\epsilon)^n$.
Moreover, $L(\Gamma_n)$ is the Lie algebra of all $n\times n$
upper triangular matrices whose super-diagonals have equal entries
and all super-diagonals at even distance from the main diagonal
are zero. In particular, $L(f)$ is abelian of dimension $n/2$ if
$n$ is even and $(n-1)/2$ if $n$ is~odd.
\end{lemma}

\begin{proof} Clearly $S_n$ and $J_n(\epsilon)'$ are polynomials in each other,
so their centralizers coincide. Further, we know that
$[L(\Gamma_n), S_n]=0$, so $L(\Gamma_n)$ is contained in the
abelian Lie algebra $C_{\mathfrak{gl}(n)}(J_n(\epsilon)')$ of all
upper triangular matrices whose super-diagonals have equal
entries. From the condition $D'\Gamma_n+\Gamma_nD=0$ we see that
$D\in L(\Gamma_n)$ if and only if all super-diagonals at even
distance from the main diagonal are zero.
\end{proof}



\begin{lemma}
Suppose that $f$ has at least two indecomposable components with
Gram matrices $\Gamma_r,\Gamma_s$, with $r,s>1$ of the same
parity. Then $L(f)$ is not reductive.
\end{lemma}


\begin{proof} Suppose first that $f$ has Gram matrix $T=\Gamma_r\oplus\Gamma_s$, with $1<r\leq s$ of the same
parity, relative to a basis $B$ of $V$. In this case $M\in L(T)$
if and only if
$$
M=\left(%
\begin{array}{cc}
  C & P \\
  Q & D \\
\end{array}%
\right),
$$
where $C\in L(\Gamma_r)$ and $D\in L(\Gamma_s)$ are as described
in Lemma \ref{lega}, $P$ and $Q$ are upper triangular with equal
entries along each diagonal, all but the last $r$
diagonals of $P$ are 0, and these are equal, up to an alternating sign, to
those of~$Q$, where the sign of the first
non-zero diagonal of $P$ is negative.

For instance, if $r=3$ and $s=5$, the matrices $M$ have the form
$$
\left( \begin{array}{ccc|ccccc}
0 & a & 0 & 0 & 0 & -d& h & -g \\
0 & 0 & a & 0 & 0 & 0 & -d& h \\
0 & 0 & 0 & 0 & 0 & 0 & 0 & -d \\\hline
d & h & g & 0 & b & 0 & c & 0 \\
0 & d & h & 0 & 0 & b & 0 & c \\
0 & 0 & d & 0 & 0 & 0 & b & 0 \\
0 & 0 & 0 & 0 & 0 & 0 & 0 & b \\
0 & 0 & 0 & 0 & 0 & 0 & 0 & 0
\end{array}\right).
$$

The elementary divisors of the asymmetry $\sigma$ of $f$ are
$(X-1)^r$ and $(X-1)^s$. Defining $ U_i=\ker(\sigma-1_V)^i $ for
$0\le i\le s$, we obtain a chain of subspaces
\begin{equation}\label{chain}
0 = U_0 \subseteq U_1 \subseteq U_2 \subseteq \cdots \subseteq U_s = V.
\end{equation}

The fact that $\sigma$ commutes with all elements of $L(f)$ implies that all these subspaces are $L(f)$-invariant. It follows that
$$
J = \{ x\in L(f) \,|\, x(U_i)\subseteq U_{i-1} \text{ for } 1\le i\le s\}
$$
is a nilpotent (hence solvable) ideal of $L(f)$. We will show that $J$ is not central.

Let $B=\{v_1,\ldots,v_r,w_1,\ldots,w_s\}$. Looking at the matrix
form of $\sigma$ we see that the chain (\ref{chain}) is actually
$$
\{0\} \subseteq \mathrm{span}\{v_1,w_1\} \subseteq
\mathrm{span}\{v_1,v_2,w_1,w_2\} \subseteq \cdots \subseteq V.
$$
Let $X\in L(T)$ be defined by $C=J_r(0)'$, $D=0$, $Q=0$, and let
$Y\in L(T)$ be defined by $C=0$, $D=0$, and
$$
Q = \left( \begin{matrix}
I_{r}\\
0_{(s-r)\times r}
\end{matrix} \right).
$$
Let $x$ and $y$ be the endomorphisms represented by $X$ and $Y$
relative to $B$. It is clear that $x\in J$. Moreover,
$[x,y]v_2=-w_1\ne 0$, so $J\nsubseteq Z(L(f))$, which shows that
$L(f)$ is not reductive.

The general case follows easily from this one.
\end{proof}

\begin{lemma}\label{oca1}
If $f$ has at least one indecomposable component with Gram matrix
$\Gamma_r$, $r\geq 3$ odd, and at least one with Gram matrix
$A_2(1)$, then $L(f)$ is not reductive.
\end{lemma}

\begin{proof} Assume first that $V$ has a basis
$B=\{v_1,v_2,\ldots,v_r,w_1,w_2,w_3,w_4\}$, with $r\geq 3$ odd,
relative to which $f$ has Gram matrix $\Gamma_r\oplus A_2(1)$.

Let $W=\ker(\sigma-1_V)$. Then $J = \{ x\in L(f) \,|\, xV\subseteq
W \text{ and } xW=0 \} $ is an abelian ideal of $L(f)$. Let
$x,y\in\gl(V)$ have matrices relative to $B$ respectively equal to
$$
\left( \begin{matrix}
0_{r\times r} & e_{11} \\
-e_{3r} & 0_{4\times 4} \\
\end{matrix} \right) \quad \text{and} \quad \left( \begin{array}{c|cccc}
0_{r\times r}   & 0 & 0 & 0 & 0\\\hline
0               & 1 & 0 & 0 &  0\\
0               & 0 & 1 & 0 & 0\\
0               & 0 & 0 & -1    & 0\\
0               & 0 & 0 & 0 & -1
\end{array} \right).
$$

A simple computation shows that both $x$ and $y$ lie in $L(f)$.
Looking at the matrix of $\sigma-1_V$ relative to $B$ we see that
$W=\mathrm{span}\{v_1,w_2,w_3\}$, whence $x\in J$. Finally we have
$[y,x]v_r=w_3$, which implies $J\nsubseteq Z(L(f))$.

The general case follows easily from this one.
\end{proof}

\begin{lemma}\label{oca2}
Suppose that the Gram matrix of $f$ relative to a basis $B$ of $V$
is $T=I_r \oplus A_2(1)$. Then $L(f)$ is not reductive.
\end{lemma}

\begin{proof}
A computation reveals that the elements of $L(T)$ are those of the
form
$$
\left( \begin{array}{c|c}
C & \begin{matrix}
-b_1   & 0      &0      & -a_1 \\
-b_2   & 0      &0      & -a_2\\
\vdots & \vdots &\vdots & \vdots\\
-b_r   & 0      &0      & -a_r
\end{matrix} \\\hline
\begin{matrix}
0 & 0 & \cdots & 0\\
a_1 & a_2 & \cdots & a_r \\
b_1 & b_2 & \cdots & b_r \\
0 & 0 & \cdots & 0
\end{matrix} & \begin{matrix}
c   & 0 & 0 & -2g   \\
e   & c & 2g& g     \\
h   & 2h& -c& -e    \\
-2h& 0  & 0 & -c
\end{matrix}
\end{array} \right)
$$
with $C\in M_r(F)$ skew-symmetric. Then
$$
J = \left\lbrace \left( \begin{array}{c|c}
0 & \begin{matrix}
-b_1   & 0      &0      & -a_1 \\
-b_2   & 0      &0      & -a_2\\
\vdots & \vdots &\vdots & \vdots\\
-b_r   & 0      &0      & -a_r
\end{matrix} \\\hline
\begin{matrix}
0 & 0 & \cdots & 0\\
a_1 & a_2 & \cdots & a_r \\
b_1 & b_2 & \cdots & b_r \\
0 & 0 & \cdots & 0
\end{matrix} & \begin{matrix}
0   & 0 & 0 & 0 \\
e   & 0 & 0 & 0     \\
0   & 0 & 0 & -e    \\
0   & 0 & 0 & 0
\end{matrix}
\end{array} \right) :a_i,b_i,e\in F \right\rbrace
$$
is a non-central ideal of $L(T)$. Since $[J,J]$ is one
dimensional, $J$ is solvable.
\end{proof}

\begin{lemma}\label{oca3}
Suppose that the Gram matrix of $f$ relative to a basis $B$ of $V$
is $T=I_r\oplus \Gamma_s$ with $s>1$ odd. Then $L(f)$ is reductive
if and only if $r=1$, in which case $L(f)$ is abelian of dimension
$(s+1)/2$.
\end{lemma}

\begin{proof}
The matrices of $L(T)$ are those of the form
$$
\left( \begin{array}{c|c}
C & \begin{matrix}
0       & \cdots & 0        & a_1\\
\vdots  & \ddots & \vdots   & \vdots\\
0       & \cdots & 0        & a_r
\end{matrix}\\\hline
\begin{matrix}
-a_1    & \cdots & -a_r \\
0       & \cdots & 0 \\
\vdots  & \ddots & \vdots\\
0       & \cdots & 0
\end{matrix} & D
\end{array} \right),
$$
with $C\in L(I_r)$ (an orthogonal Lie algebra) and $D\in
L(\Gamma_s)$. In fact, $L(T)\cong L(I_r)\ltimes J$, where $J$ is
the abelian ideal given by
$$
J=\left\lbrace \left( \begin{array}{c|c}
C & P\\\hline
Q & D
\end{array} \right) \in L(T): C = 0 \right\rbrace.
$$

If $r=1$ then $L(I_r)=0$ and $L(f)$ is abelian. If $r\ge 3$, the
action of $L(I_r)$ on $J$ is not trivial, so $L(f)$ is not
reductive in this case.
\end{proof}

By combining the above results with those of \S\ref{secuno} as
well as Theorem \ref{sergei} we obtain

\begin{prop}\label{sumario1} Suppose that $\ell\neq 2$ and $f$ is non-degenerate with unipotent
asymmetry $\si$. Then $L(f)$ is absolutely reductive if and only if $\si$
cyclic, in which case $L(f)$ is abelian, or $\si=1_V$,
in which case $f$ is symmetric, or $\si$ has elementary divisors
$(X-1)^2,(X-1)^2$, in which case $L(f)\cong\gl(2)$, or $\si$ has elementary divisors
$(X-1)^{2m+1},(X-1)$ with $m\geq 1$, in which case $L(f)$ is abelian.
\end{prop}

\begin{proof} It only remains to see that $L(f)$ is absolutely
reductive when $f$ is symmetric. This follows from \cite{B},
Chapter 1, \S 6, Exercise 26. See \cite{CS}, \S 8 and \S9, for
full details.
\end{proof}

\begin{lemma}
If $f$ admits $T=A_1(-1)\oplus\cdots\oplus A_1(-1)\oplus\Gamma_r$
as Gram matrix for some even $r$ then $L(f)$ is not reductive.
\end{lemma}

\begin{proof}
By reordering basis vectors, we obtain a basis $B$ of $V$ relative
to which $f$ has Gram matrix
$$
\left( \begin{matrix}
0 & I_s\\
-I_s & 0
\end{matrix} \right) \oplus \Gamma_r.
$$

A few calculations show that the elements of $L(T)$ are those of
the form
$$
\left( \begin{array}{c|c}
C & \begin{matrix}
0       & \cdots & 0        & -b_1\\
\vdots  & \ddots & \vdots   & \vdots\\
0       & \cdots & 0        & -b_s\\
0       & \cdots & 0        & a_1\\
\vdots  & \ddots & \vdots   & \vdots\\
0       & \cdots & 0        & a_r
\end{matrix}\\\hline
\begin{matrix}
a_1     & \cdots & a_r      & b_1   & \cdots & b_r \\
0       & \cdots & 0        & 0     & \cdots & 0\\
\vdots  & \ddots & \vdots   & \vdots& \ddots & \vdots \\
0       & \cdots & 0        & 0     & \cdots & 0
\end{matrix} & D
\end{array} \right),
$$
where $C\in\mathfrak{sp}(2s)$ and $D\in L(\Gamma_r)$. Then
$L(f)\simeq \mathfrak{sp}(2s)\ltimes J$, where
$$
J = \left\lbrace \left( \begin{array}{c|c}
C & P\\\hline
Q & D
\end{array} \right) \in L(T) : C=0 \right\rbrace\unlhd L(T).
$$

This ideal is solvable because $[J,J]$ is one-dimensional.
However, the action of $\mathfrak{sp}(2s)$ on $J$ is not trivial.
Therefore $L(f)$ is not reductive.
\end{proof}

Much as above, we may now derive the following

\begin{prop}\label{sumario2} Suppose that $\ell\neq 2$ and $f$ is non-degenerate with asymmetry
$\si$, where $\si+1_V$ is nilpotent. Then $L(f)$ is absolutely reductive if
and only if $\si$ cyclic, in which case $L(f)$ is abelian, or $\si=-1_V$,
in which case $f$ is skew-symmetric.
\end{prop}

\begin{proof} It only remains to see that $L(f)$ is absolutely
reductive when $f$ is skew-symmetric. This follows from \cite{B},
Chapter 1, \S 6, Exercise 25. See \cite{CS}, \S 10, for full
details.
\end{proof}

As a consequence of Propositions  \ref{jodd3}, \ref{p23},
\ref{zxc}, \ref{sumario1} and \ref{sumario2}, we obtain

\begin{theorem}\label{todjur} Suppose $F$ is algebraically closed
of characteristic not 2. Then $L(f)$ is reductive if and only if
either $f=0$, in which case $L(f)=\gl(V)$, or else $f$ admits as
Gram matrix a direct sum of matrices of the following types, in
which case $L(f)$ is isomorphic to the direct sum of the Lie
algebras associated to these matrix summands, as indicated below.
Moreover, in this case, at last one summand occurs; there exists
at most one summand of each type; if a summand of type $\lambda$
occurs, where $\lambda\in F$ and $\lambda\neq \pm 1$, then no
summand of type $\lambda^{-1}$ exists.

\noindent{\sc Type 0.}

$A=\left(%
\begin{array}{cc}
  0 & J_m \\
  I_m & 0 \\
\end{array}%
\right)$, $L(A)$ abelian of dimension $m$;

$B=\left(%
\begin{array}{cc}
  0 & 0 \\
  I_m & 0 \\
\end{array}%
\right)$, $L(B)\cong\gl(m)$, $m>1$.

\noindent{\sc Type $\lambda$, where $\lambda\in F$ and
$\lambda\neq \pm 1$.}

$A=\left(%
\begin{array}{cc}
  0 & J_m(\lambda) \\
  I_m & 0 \\
\end{array}%
\right)$, $L(A)$ abelian of dimension $m$;

$B=\left(%
\begin{array}{cc}
  0 & \lambda I_m \\
  I_m & 0 \\
\end{array}%
\right)$, $L(B)\cong\gl(m)$, $m>1$.

\noindent{\sc Type 1.}

$A=\Gamma_m$, $m$ odd, $L(A)$ abelian of dimension $(m-1)/2$;

$B=I_m$, $L(B)\cong\so(m)$, $m>2$;

$C=\left(%
\begin{array}{cc}
  0 & I_2 \\
  J_2(1) & 0 \\
\end{array}%
\right)$, $L(C)\cong\gl(2)$;

$D=\Gamma_m\oplus \Gamma_1$, $m$ odd, $L(A)$ abelian of dimension
$(m+1)/2$.

\noindent{\sc Type $-1$.}

$A=\Gamma_m$, $m$ even, $L(A)$ abelian of dimension $m/2$;

$B=\left(%
\begin{array}{cc}
  0 & I_m \\
  -I_m & 0 \\
\end{array}%
\right)$, $L(B)\cong\sy(2m)$.

\end{theorem}

\section{Simplicity and semisimplicity of $L(f)$}

\begin{theorem}
\label{fsemi} The Lie algebra $L(f)$ is semisimple if and only if
$\ell\neq 2$, $f$ is non-degenerate and $f=f_1\perp f_{-1}$, where
$f_1$ is symmetric, $\dm(V_1)\neq 2$ and $f_{-1}$ is
skew-symmetric.
\end{theorem}

\begin{proof} If the stated conditions are satisfied then $L(f)$
is known to be semisimple  (see \cite{B}, Chapter 1, \S 6,
Exercises 25 and 26). Suppose conversely that $L(f)$ is
semisimple.

\noindent{\bf Step 1.} $V_{\mathrm{odd}}=0$.

Suppose, if possible, that  $V_{\mathrm{odd}}\neq 0$. Then
$V_{\mathrm{odd}}=\mathrm{Rad}(f)$ by Lemma \ref{jodd}, so $f=0$
by Lemma \ref{joddunoymedio}, whence $L(f)=\gl(V)$ is semisimple,
a contradiction.

\noindent{\bf Step 2.} $V_{\mathrm{even}}=0$.

As indicated in \S\ref{sectermnot}, we have $L(f)\cong
L(f_{\mathrm{even}})\oplus L(f_{\mathrm{ndeg}})$. Both summands
must be semisimple. We see from Proposition \ref{p23} that
$V_{\mathrm{even}}=0$.

\noindent{\bf Step 3.} The asymmetry $\sigma$ of $f$ satisfies
$\sigma=\sigma^{-1}$.

This is clear from the fact, shown in Lemmas \ref{lem21a} and
\ref{lem22a}, that $\sigma-\sigma^{-1}\in Z(L(f))$.

\noindent{\bf Step 4.} $\ell\neq 2$ and $f=f_1\perp f_{-1}$, where
$f_1$ is symmetric and $f_{-1}$ is skew-symmetric.

By above $\sigma$ satisfies the polynomial $X^2-1$. We have
$\ell\neq 2$, for otherwise $1_V\in Z(L(f))$. The conclusion now
follows at once from Lemma \ref{lem83}.

\noindent{\bf Step 5.} $\dm(V_1)\neq 2$.

If $\dm(V_1)=2$ then $L(f_1)$ is non-zero and abelian.
\end{proof}

Taking into account \cite{B}, Chapter 1, \S 6, Exercises 25 and
26, we may now derive the following

\begin{theorem}
\label{panda}
 The Lie algebra $L(f)$ is simple if and only if
$\ell\neq 2$, $f$ is non-degenerate and exactly  one of the
following conditions hold:

$\bullet$ $f$ is skew-symmetric;

$\bullet$ $f$ is symmetric, $\dm(V)>2$, and if $\dm(V)=4$ then the
discriminant of $f$ is not a square in $F$.

$\bullet$  $f=f_1\perp f_{-1}$, where $f_1$ is symmetric, $f_{-1}$
is skew-symmetric, $\dm(V_1)=1$ and $\dm(V_{-1})>0$.

\end{theorem}

\section{When is $L(f)$ isomorphic to some $\sl(m)$}

\begin{theorem}
\label{tip}
 Suppose $\ell\neq 2$ and let $m=\dm(V)$. Then $\sl(n)$
isomorphic to $L(f)$ only when $f$ is non-degenerate and one of following cases occurs: $n=2$, $m=2$ and $f$ is  skew-symmetric;
$n=2$, $m=3$ and $f$ is a suitable symmetric bilinear form;
$n=4$, $m=6$ and $f$ is a suitable symmetric bilinear form; $n=2$, $f=f_1\perp f_{-1}$ with $f_1$ symmetric,
$f_{-1}$ skew-symmetric, $\dm(V_1)=1$ and $\dm(V_{-1})=2$.
\end{theorem}

\begin{proof} Suppose $\sl(n)$ is isomorphic to some $L(f)$. We may assume that
$F$ is algebraically closed. In this case we have a description of
all bilinear forms $f$ whose associated Lie algebra is reductive.
Of these, the only $f$ such that $L(f)$ is simple modulo its
center, as $\sl(n)$ is, are in fact those for which $L(f)$ is
already simple, as given in Theorem \ref{panda}.

It remains to see when $\sl(n)$ is isomorphic to $\sy(2d)$,
$\so(2d+1)$ or $\so(2d)$. Care must be taken as $\ell\neq 2$ is
arbitrary.

Write $L$ for $\sl(n)$ and $M$ for any of the other 3 Lie
algebras. Let $H$ be the diagonal subalgebra of $M$. Then $ad\,
h:M\to M$ is diagonalizable for every $h\in H$, i.e., $H$ is a
toral subalgebra of $M$, and there is $h\in H$ such that
$C_M(h)=H$. Thus there is a toral subalgebra $T$ of $L$ and $t\in
T$ such that $C_L(t)=T$. Since $T$ is toral, $ad\,t:L\to L$ is
diagonalizable. It follows easily from the JC-decomposition of $t$
in $\gl(n)$ that $t$ itself is diagonalizable. Being toral, $T$ is
abelian. Thus all eigenvalues of $t$ are distinct, whence
$\dm(T)=n-1$. But $\dm(H)=d$, so $n=d+1$. Now if $M=\sy(2d)$ or
$M=\so(2d+1)$ then $n^2-1=2d^2+d$, so $d=1$, while if $M=\so(2d)$
then $n^2-1=2d^2-d$, so $d=3$, as required.

The actual existence of an isomorphism is obvious in all but the
third case, when the action of $\sl(4)$ on the second exterior
power of its natural module yields an $\sl(4)$-invariant
non-degenerate symmetric bilinear form.
\end{proof}

\begin{note} The ``suitable" symmetric bilinear forms cited in Theorem \ref{tip} are
not totally arbitrary. Indeed, let $F=\mathbb{R}$ and let $L(f)$
be the Lie algebra of a definite (positive or negative) symmetric
bilinear form $f$. Then $ad\,x$ is semisimple for all $x\in L(f)$,
so $L(f)$ is not isomorphic to $\sl(n)$ for any $n\geq 2$.
\end{note}

\begin{theorem}\label{pe4} If $\ell=2$ and $n>2$ then $\sl(n)$ is
not isomorphic to any $L(f)$.
\end{theorem}

\begin{proof} Suppose, if possible, that $L(f)$ is isomorphic to $\sl(n)$.

\noindent{\bf Step 1.} $V_{\mathrm{odd}}=0$.

We know from \cite{B}, Chapter 1, \S 6, Exercise 24, that $\sl(n)$
has at most one intermediate ideal, which is of dimension 1. If
the ideal $L(f)\cap\sl(V)$ of $L(f)$ had dimension 0 or 1 then
$\dm(L(f))\leq 2$, a contradiction. Thus $L(f)$ is included in
$\sl(V)$, so $V_{\mathrm{odd}}=0$ by Lemma \ref{jodd}.

\noindent{\bf Step 2.} $V_{\mathrm{even}}=0$.

This is clear from the description of $L(f_{\mathrm{even}})$ and
the structure of $\sl(n)$.

\noindent{\bf Step 3.} The asymmetry $\si$ of $f$ is unipotent.

Extending scalars, we may assume that all eigenvalues of $\si$ are
in $F$. The result now follows from the structure of $\sl(n)$ and
the description of the Lie algebra of $f$ restricted to
$V_\lambda\oplus V_{\lambda^{-1}}$ for any eigenvalue $\lambda\neq
1$ of $\si$.

\noindent{\bf Step 4.} $\si=1_V$.

Since $\si,1_V\in Z(L(f))$, $\si$ is unipotent and
$\dm(Z(\sl(n)))=1$, we see that~$\si=1_V$.

\medskip

By Step 4, $f$ is symmetric. Thus $L(f)$ is not perfect, by
\cite{CS}, Theorem 7.8. Since $\sl(n)$ is perfect by \cite{B},
Chapter 1, \S 6, Exercise 24, we reached a contradiction.
\end{proof}

\begin{prop}\label{pe5} If $\ell=2$ and $m=\dm(V)$ then $\sl(2)$ is isomorphic to
$L(f)$ only when $m=2$ and $f$ is non-degenerate and alternating.
\end{prop}

\begin{proof} It is clear that $\sl(2)$ is isomorphic to $L(f)$ when
$m=2$ and $f$ is non-degenerate and alternating.

Suppose, conversely, that $\sl(2)$ is isomorphic to $L(f)$. It follows Lemma \ref{jodd} that $V_{\mathrm{odd}}$ cannot have
any indecomposable components $J_n$ with $n>3$.

Suppose, if possible, that $V_{\mathrm{odd}}$ has at least one
component $J_3$. Since $\sl(2)$ is 3-dimensional the multiplicity
of $J_3$ must be 1. Since $L(J_3)$ is 2-dimensional, at least one
of $V_{\mathrm{even}}$ and $V_{\mathrm{ndeg}}$ must be non-zero.
But $1_{V_{\mathrm{even}}}\in L(f_\mathrm{even})$ and
$1_{V_{\mathrm{ndeg}}}\in L(f_\mathrm{ndeg})$, which together with
$L(J_3)$ already produce a 3-dimensional abelian subalgebra of
$L(f)$, a contradiction.

Suppose next, if possible, that $V_{\mathrm{odd}}$ has at least
one component $J_1$. More than one such component would yield a
Lie algebra of dimension $\geq 4$, so there is only one of them.
Thus $\mathrm{Rad}(f)$ is 1-dimensional. Let $a$ be the
codimension of $\mathrm{Rad}(f)$ in $V$. Then $K=\{x\in\gl(V)\,|\,
x \mathrm{Rad}(f)=0, xV\subseteq \mathrm{Rad}(f)\}$ is an abelian
ideal of $L(f)$ of dimension $a$ having trivial intersection with
$L(f_\mathrm{odd})\oplus L(f_\mathrm{even})\oplus
L(f_\mathrm{ndeg})$. Since the latter has dimension at least 2,
then necessarily $a=1$, $V_{\mathrm{even}}=0$ and
$V_{\mathrm{ndeg}}$ is 1-dimensional. In this case $L(f)$ is
indeed 3-dimensional but its derived algebra is not central, hence
not isomorphic to $\sl(2)$. We conclude that $V_{\mathrm{odd}}=0$.

Now $L(f_\mathrm{even})$ is isomorphic to the centralizer of a
nilpotent matrix $A$ in $\gl(k)$. This matrix cannot be cyclic,
for otherwise $L(f_\mathrm{even})$ is abelian. But if $A$ has at
least two elementary divisors we easily see that the dimension of
$L(f_\mathrm{even})$ is at least~4. We conclude that
$V_{\mathrm{even}}=0$.

Thus $f$ is non-degenerate. Let $\si$ be the asymmetry of $f$.
Without loss of generality we may assume that $f$ is algebraically
closed. If $\si$ had any eigenvalue $\lambda\neq 1$ then applying
to $V_\lambda\oplus V_{\lambda^{-1}}$ the same argument used for
$V_\mathrm{even}$ would yield a contradiction.

We infer that $\si$ is unipotent. Since $Z(\sl(2))$ is
1-dimensional, we see as above that $\si=1_V$, so $f$ is
symmetric. Then $L(f)$ has dimension $m(m+1)/2$, so $m=2$. If $f$
is not alternating then $L(f)$ is 3-dimensional but its derived
algebra is not central. Therefore $f$ is alternating.
\end{proof}

\medskip

\noindent{\bf Acknowledgements.} We thank D. Djokovic for fruitful
discussions held long ago, A. Premet for a useful reference, and
the referee for a careful reading of the paper and valuable
suggestions.

\end{document}